\documentclass[leqno,11pt,twoside]{amsart}
\usepackage{etex}

\usepackage{geometry}

\usepackage{times}

\usepackage[all]{xy} 

\usepackage{amsmath, amssymb, amsfonts, latexsym, mdwlist, amsthm,amscd}
\usepackage{caption, subcaption}

\usepackage{url}

\usepackage{verbatim}
\usepackage{mathtools}

\usepackage{pgf,tikz}
\usepackage{tikz-cd}
\usetikzlibrary{calc,matrix,patterns,shadings,backgrounds,fit}
\pgfdeclarelayer{myback}
\pgfsetlayers{myback,background,main}
\usepackage{stmaryrd}
\usepackage{diagbox}

\usepackage[bookmarks, colorlinks, breaklinks, pdftitle={Brill-Noether for generic abelian surfaces},
pdfauthor={Arend Bayer, Chunyi Li}]{hyperref}
\hypersetup{linkcolor=blue,citecolor=blue,filecolor=black,urlcolor=blue}

\usepackage{paralist}
\setdefaultenum{(a)}{(i)}{}{}
\usepackage{enumitem} 

\def\C{\ensuremath{\mathbb{C}}}
\def\D{\ensuremath{\mathbb{D}}}

\def\P{\ensuremath{\mathbb{P}}}
\def\Q{\ensuremath{\mathbb{Q}}}
\def\R{\ensuremath{\mathbb{R}}}
\def\Z{\ensuremath{\mathbb{Z}}}

\def\alg{\mathrm{alg}}

\def\ch{\mathop{\mathrm{ch}}\nolimits}

\def\Coh{\mathop{\mathrm{Coh}}\nolimits}

\def\dim{\mathop{\mathrm{dim}}\nolimits}

\def\ext{\mathop{\mathrm{ext}}\nolimits}
\def\Ext{\mathop{\mathrm{Ext}}\nolimits}

\def\GL{\mathop{\mathrm{GL}}\nolimits}
\def\Hilb{\mathrm{Hilb}}
\def\Hom{\mathop{\mathrm{Hom}}\nolimits}

\def\Ker{\mathop{\mathrm{Ker}}\nolimits}

\def\mod{\mathop{\mathrm{mod}}\nolimits}

\def\NS{\mathop{\mathrm{NS}}\nolimits}

\def\Pic{\mathop{\mathrm{Pic}}}

\def\rk{\mathop{\mathrm{rk}}}
\def\RlHom{\mathop{\mathbf{R}\mathcal Hom}\nolimits}

\DeclarePairedDelimiter\floor{\lfloor}{\rfloor}
\def\Stab{\mathop{\mathrm{Stab}}\nolimits}
\def\into{\ensuremath{\hookrightarrow}}
\def\onto{\ensuremath{\twoheadrightarrow}}

\def\blank{\underline{\hphantom{A}}}

\def\Db{\mathrm{D}^{b}}

\def\VrdH{V^r_d(\abs{H})}

\def\star{(*)}

\newcommand\TFILTB[3]{%
\xymatrix@=1pc{
{0 = {#1}_0} \ar[rr]&&
{{#1}_1} \ar[rr]\ar[ld] &&
{{#1}_2} \ar[r]\ar[ld] &
{\cdots} \ar[r] & { {#1}_{#3-1}} \ar[rr] &&
{{#1}_{#3} = {#1}} \ar[ld]
\\
& *{{#2}_1} \ar@{.>}[ul] &&
{{#2}_2} \ar@{.>}[ul] & &&&
{{#2}_{{#3}}} \ar@{.>}[ul]
}}

\def\abs#1{\left\lvert#1\right\rvert}

\newcommand\stv[2]{\left\{#1\,\colon\,#2\right\}}

\makeatletter
\newtheorem*{rep@theorem}{\rep@title}
\newcommand{\newreptheorem}[2]{%
\newenvironment{rep#1}[1]{%
 \def\rep@title{#2 \ref{##1}}%
 \begin{rep@theorem}}%
 {\end{rep@theorem}}}
\makeatother

\newtheorem{Thm}{Theorem}[section]
\newreptheorem{Thm}{Theorem}
\newtheorem{Prop}[Thm]{Proposition}

\newtheorem{Lem}[Thm]{Lemma}

\newtheorem{Cor}[Thm]{Corollary}
\newreptheorem{Cor}{Corollary}

\newreptheorem{Con}{Conjecture}

\newtheorem{thm-int}{Theorem}

\theoremstyle{definition}
\newtheorem{Def-s}[Thm]{Definition}
\newtheorem{Def}[Thm]{Definition}
\newtheorem{Rem}[Thm]{Remark}

\def\C{\ensuremath{\mathbb{C}}}
\def\D{\ensuremath{\mathbb{D}}}

\def\P{\ensuremath{\mathbb{P}}}
\def\Q{\ensuremath{\mathbb{Q}}}
\def\R{\ensuremath{\mathbb{R}}}
\def\Z{\ensuremath{\mathbb{Z}}}

\def\cE{\ensuremath{\mathcal E}}
\def\cF{\ensuremath{\mathcal F}}

\def\cL{\ensuremath{\mathcal L}}

\def\cO{\ensuremath{\mathcal O}}
\def\cP{\ensuremath{\mathcal P}}

\def\cT{\ensuremath{\mathcal T}}

\def\cW{\ensuremath{\mathcal W}}

\def\vv{\ensuremath{\mathbf v}}

\def\iff{\; \Longleftrightarrow \;}

\usepackage{todonotes}

\def\Halg{H^*_{\alg}(X, \Z)}
\def\HalgR{H^*_{\alg}(X, \R)}
\def\Zab{Z_{\alpha, \beta}}
\def\sab{\sigma_{\alpha, \beta}}

\begin{document}

\title[Brill-Noether theory for curves on generic abelian surfaces]{Brill-Noether theory for curves on generic abelian surfaces}

\author{Arend Bayer}
\address{School of Mathematics and Maxwell Institute,
University of Edinburgh,
James Clerk Maxwell Building,
Peter Guthrie Tait Road, Edinburgh, EH9 3FD,
United Kingdom}
\email{arend.bayer@ed.ac.uk}
\urladdr{http://www.maths.ed.ac.uk/~abayer/}

\author{Chunyi Li}
\address{School of Mathematics and Maxwell Institute,
University of Edinburgh,
James Clerk Maxwell Building,
Peter Guthrie Tait Road, Edinburgh, EH9 3FD,
United Kingdom}
\email{chunyi.li@ed.ac.uk}
\urladdr{https://sites.google.com/site/chunyili0401/}


\begin{abstract}
We completely describe the Brill-Noether theory for curves in the primitive linear system on generic abelian surfaces, in the following sense: given integers $d$ and $r$, consider the variety $V^r_d(\abs{H})$ parametrizing curves $C$ in the primitive linear system $\abs{H}$ together with a torsion-free sheaf on $C$ of degree $d$ and $r+1$ global sections. We give a necessary and sufficient condition for this variety to be non-empty, and show that it is either a disjoint union of Grassmannians, or irreducible. Moreover, we show that, when non-empty, it is of expected dimension.
 
This completes prior results by Knutsen, Lelli-Chiesa and Mongardi.
\end{abstract}

\dedicatory{Dedicated to Yuri Ivanovich Manin on the occasion of his 80th birthday.}


\maketitle

\setcounter{tocdepth}{1}
\tableofcontents
\def\mhk{M^h_k} 
\def\lrd{V^r_d(\abs{H})}
\def\ext{\mathrm{ext}}
\def\Ext{\mathrm{Ext}}

\section{Introduction}

By Lazarsfeld's celebrated result \cite{Lazarsfeld:BN-Petri}, a smooth curve in the primitive linear system of a generic K3 surface is Brill-Noether general (in the strongest possible sense, see \cite[Theorem 1.1]{wallcrossing-BrillNoether}). The corresponding question for abelian surfaces is much more subtle. In the present article, we completely determine when the Brill-Noether locus for the entire primitive linear system on generic abelian surfaces is non-empty, and show that, when non-empty, it is of the expected dimension. We show that the Brill-Noether locus is either irreducible or a disjoint union of Grassmannians.
In particular, unlike for K3 surfaces our condition provides many examples with negative
Brill-Noether numbers where \emph{some} of the curves in the primitive linear system are not
Brill-Noether general, in the sense of carrying torsion-free sheaves of prescribed degree and number
of global sections.
This completes previous work by Knutsen, Lelli-Chiesa and Mongardi \cite{KLM}. 

Let $X$ be an abelian surface whose N\'eron-Severi group $\NS(X)$ is generated by the class of an ample line bundle $H$.
Let $g = \frac{H^2}2 + 1$ be the arithmetic genus of curves in $\abs{H}$.
Given integers $r \ge 1$ and $d\geq 1$, write $\chi = d+1-g$, and recall that the Brill-Noether number is given by $\rho(r,d,g) = g - (r+1)(r+1-\chi)$. Let $V^r_d(\abs{H})$ denote the Brill-Noether locus parametrizing curves $C$ in the linear system $\abs{H}$ together with a pure sheaf $L$ supported on $C$ with $c_1(L) - H$, with $\chi(L) = \chi$, and $h^0(L) = r+1$; this includes, of course, smooth curves in $\abs{H}$ equipped with a complete linear system $g^r_d \colon C \to \P^r$ of degree $d$.  Our main result is the following:

\begin{Thm}
Assume $\chi \neq 0$.
The Brill-Noether locus $V^r_d(\abs{H})$ is non-empty if and only if 
\[ \rho + g - 2 \ge D\abs{\chi}-D^2, \]
where $D$ denotes the remainder of division of $r+1$ by $\abs{\chi}$.
Moreover, when it is non-empty, it is generically smooth and of expected dimension $\rho + g - 2$.

When the above inequality is strict, then $\VrdH$ is irreducible. Otherwise, it is a disjoint union of $\left(\frac{g-1}{\chi}\right)^2$ Grassmannians.
\label{main1}
\end{Thm}
In \cite[Theorems 1.4 and A.1]{KLM}, the authors showed that the last condition is necessary for non-emptiness, and that in this case the Brill-Noether locus has a component of expected dimension; they also showed that this condition is sufficient when $d \ge r(r+1)$, and further implies the existence of \emph{smooth} curves with $g^r_d$. In other words, in addition to their results we show that this necessary condition is also sufficient for $d < r(r+1)$, and we determine when the Brill-Noether locus is irreducible.

The fact that $\VrdH$ has expected dimension implies in particular that a generic curve in $\abs{H}$ is Brill-Noether general, in the sense that it has no line bundle with $\rho < 0$. This was first proved by Paris \cite{Paris:Petri}, along with the Petri property, under the same assumption $\chi \neq 0$.

In fact, a similar statement holds for certain moduli spaces of vector bundles. 
Let $\vv \in H^*(X)$ be a class of the form
$\vv = (k, c_1(H), \chi)$ for some integers $k \in H^0(X, \Z) = \Z$ and
$\chi \in H^4(X, \Z) = \Z$. Then the moduli space $M_H(\vv)$ of Gieseker-stable sheaves with Chern character $v$ is smooth and irreducible of dimension $\vv^2 + 2$. 
Let $M^{r+1}_H(\vv)$ denote the subset of sheaves $E$ with $h^0(E) = r+1$.

\begin{Thm}\label{mainthm2}
Assume that $\chi < 0$, and let $r, g, D$ be as above. Then
$M^{r+1}_H(\vv)$ is non-empty if and only if
\[\vv^2-(r+1)(r+1-\chi)\ge D(-\chi)-D^2.\] 
In this case, it is irreducible and of expected dimension $\vv^2+2-(r+1)(r+1-\chi)$.
\end{Thm}

\subsection{Comparison with \cite{KLM}}
When $H^2=54$, in other word, $g=28$,  the following table lists the existence and emptiness result for curves in $V^r_d(\abs{H})$ for $20\leq d\leq 26$ and $r\geq 1$.\\

\begin{tabular}{|l||*{7}{c|}}\hline
\backslashbox{degree}{r(section)}
&\makebox[3em]{1}&\makebox[3em]{2}&\makebox[3em]{3}
&\makebox[3em]{4}&\makebox[3em]{5}&\makebox[3em]{6}&\makebox[3em]{7}\\\hline\hline
20 &  BN & KLM & KLM & $\phi$ & $\phi$ & $\phi$ & $\phi$\\\hline
21 &  BN & BN & KLM & $\phi$ & $\phi$ & $\phi$ & $\phi$\\\hline
22 &  BN & BN & KLM & KLM & $\phi$ & $\phi$ & $\phi$\\\hline
23 &  BN & BN & KLM & KLM & $\phi$ & $\phi$ & $\phi$\\\hline
24 &  BN & BN & BN & KLM & Thm.~\ref{main1} & $\phi$ & $\phi$\\\hline
25 &  BN & BN & BN & KLM & $\Delta$ & $\phi$ & $\phi$\\\hline
26 &  BN & BN & BN & KLM & $\Delta$ & $\phi$ & $\phi$\\\hline
\end{tabular}
\begin{center}
Table: (non-)emptiness of $V^r_d(\abs{H})$.
\end{center}
The labels in each box indicate the following situations:
\begin{description}
\item[$\phi$] There is no line bundle of degree $d$ with $r+1$ section on a curve in $\abs{H}$; this follows from \cite[Theorem A.1]{KLM} or our Theorem \ref{main1}.
\item[BN] The Brill-Noether number, $\rho(g,r,d)$, is non-negative, and thus every smooth curves carries a $g^r_d$ by \cite{Kempf:Schubert,Kleim-Lak},
\item[KLM] We have $\rho(g, r, d) < 0$, but there exists a smooth curve with a $g^r_d$ by \cite[Theorem 1.4]{KLM}. 
\item[$\Delta$] $\VrdH$ is non-empty, see discussion below.
\item[Thm. 1.1] The non-emptiness of $\VrdH$ follows from Theorem \ref{main1}.
\end{description}

We now give more explanation on the cells marked `$\Delta$' . First of all, by our Theorem \ref{main1}, $V^r_d(\abs{H})$ is non-empty. On the other hand, the non-emptiness of such $V^r_d(\abs{H})$ can be deduced from \cite[Theorem 1.6 (i)]{KLM} with the same arguments as that in \cite[Example 5.15]{KLM}. For example, when $r=5$ and $d=25$,  one may let $k=2$ and $\delta =15$. It is direct to check that $\delta$, $k$, and $g$ (or $p$ in the notation of \cite{KLM}) satisfy the inequality in Theorem 1.6(i) in \cite{KLM}. By the theorem, there is a  curve $C$ in $\abs{H}$ with $15$ nodes, and its normalization $\tilde{C}$ is hyper-elliptic. In particular,  $\tilde{C}$ carries linear series with degree $10$ and rank $5$. Push-forward to $C$ produces a torsion free sheaf with degree $10+15=25$ and rank $5$.
For the box that is marked by `Thm.~1.1', the non-emptiness of $V^r_d(\abs{H})$ is due to Theorem \ref{main1} and is completely new. Interesting readers may check that Theorem 1.6 (i) in \cite{KLM} does not provide suitable nodal curves that may carry linear series as desired. 

\begin{Rem}
 Note that $V^r_d(\abs{H})$ in our setting is different from the space $\abs{L}^r_d$ defined in \cite{KLM}: for the latter, the support curve is required to be smooth. 
\label{denotion}
\end{Rem}

\subsection{Proof strategy}
We proof Theorem \ref{main1} by wall-crossing. Let $\vv = (0, H, \chi)$; then the moduli space $M_H(\vv)$ contains all pure torsion sheaves $F$ with $c_1(F) = H$ and $\chi(F) = \chi$. It can be reinterpreted as the moduli space $M_{\sigma}(\vv)$ of Bridgeland-stable objects when $\sigma$ is contained in the \emph{Gieseker-chamber} for $\vv$. We then show (for $\chi < 0$) that the first wall bounding the Gieseker-chamber destabilizes exactly those $F$ with
$h^0(F \otimes L_0) \neq 0$ for degree zero line bundle $L_0$ on $X$. In particular, there is a natural short exact sequence
\[ \cO_X \otimes H^0(F) \into F \onto F' \]
of semistable objects that is part of the Jordan-H\"older filtration of $F$. However, unlike in the case of K3 surfaces treated in \cite{wallcrossing-BrillNoether}, we may have $h^0(F') \neq 0$; in fact, this may be necessary for the extension $F$ to exist, as otherwise $\Ext^1(F', \cO_X)$ might be too small. 

Thus, we have to proceed by induction. This induction is possible since all steps remain part of the finite Jordan-H\"older filtration of $F$. On the other hand, it is precisely this induction process that leads to the slightly arithmetic nature of our results (involving the division by remainder). 

\subsection{Generalisation} All our results hold slightly more generally for any polarized abelian surface $(X, H)$ satisfying the following:
\begin{description}
\item[Assumption (*)] $H^2$ divides $H.D$ for all curve classes $D$ on $X$.
\end{description}
To simplify the presentation, we first explain our entire argument in the case of Picard rank one; we then explain in Section \ref{sect:moregeneral} how to extend the arguments to this situation.

\subsection{Related questions} In addition to \cite{KLM}, there have been a number of recent articles studying Brill-Noether loci of curves in $M_g$ carying a $g^r_d$ for negative Brill-Noether numbers, in particular \cite{Pflueger:negativeBN,Pflueger:BN-k-gonal,Jensen-Ranganathan:BN-fixed-gonality}. However, our results have no direct implication on this question, as we cannot prove that smooth curves appear in our locus $\VrdH$.

\subsection{Acknowledgments} It is a pleasure to thank Emanuel Macr{\`{\i}}, Giovanni Mongardi, Soheyla Feyzbakhsh, and Sam Payne for a number of helpful discussions and comments. Both authors are supported by ERC starting grant no. 337039 ``WallXBirGeom''.

It is a pleasure to dedicate this article to Yuri I.~Manin. He first suggested to the first author to start exploring stability conditions and wall-crossing. This article is one of many puzzle pieces of a large picture of applications of wall-crossing that has emerged since then.

\section{Background: stability conditions, moduli spaces}
\label{sect:background}

\subsection{Review: stability conditions on abelian surfaces}
Let $X$ be an abelian surface of Picard rank one; we denote by $D^b(X)$ the bounded derived category of coherent sheaves on
$X$. In this section,  we will review the description of a component of the space $\Stab(X)$ of stability conditions on $D^b(X)$ given in \cite[Section 15]{Bridgeland:K3}.

Given an object $E \in D^b(X)$ we write $\ch(E) \in H^*(X, \Z)$ for its Chern character with value in cohomology. We write $H^*_{\mathrm{alg}}(X, \Z)$ for its algebraic part, i.e.~the image of $\ch$.

Let $H$ be a line bundle as above; by abuse of notation, we will also write $H$ for its class in $\NS(X)$. Given $\beta \in \R$, we defined the $\beta$-twisted Chern character by
 \[ \ch^\beta(E) := e^{-\beta H}.\ch(E) \in H^*_{\mathrm{alg}}(X, \R) \]
and the $\beta$-twisted slope of a coherent sheaf $E \in \Coh X$ by 
\[ \mu_{H,\beta}(E) := \begin{cases}
\frac{H.\ch^\beta_1(E)}{H^2\ch_0(E)} & \text{if $\ch_0(E) > 0$} \\
+\infty & \text{if $\ch_0(E) = 0$}.
\end{cases}\]
This leads to the usual notion of slope-stability, and the construction of the following torsion pair
\begin{align*}
\cT^{\beta} & :=
\langle E \colon \text{$E$ is $\mu_{H, \beta}$-semistable with $\mu_{H, \beta}(E) > 0$ } \rangle, \\
\cF^{\beta} & :=
\langle E \colon \text{$E$ is $\mu_{H, \beta}$-semistable with $\mu_{H, \beta}(E) \le 0$} \rangle,
\end{align*}
where $\langle \cdot \rangle$ denotes the extension-closure.
This is a torsion pair in $\Coh X$. 
Following \cite{Happel-al:tilting,Bridgeland:K3}, this lets us define a new heart of a bounded t-structure in $\Db(X)$ as follows:
\[
\Coh^\beta X := \langle \cF^\beta[1], \cT^\beta \rangle =
\stv{E}{H^{-1}(E) \in \cF^\beta, H^0(E) \in \cT^\beta, H^i(E) = 0 \text{ for $i \neq 0, -1$}}.
\]

For $\alpha > 0$ and $\beta \in \R$ as before, we define the central charge $Z_{\alpha, \beta} \colon K(X) \to \C$ by
\begin{equation} \label{eqn:Zab}
Z_{\alpha, \beta}(E) := - \int_X e^{-i\alpha H} . \ch^\beta(E)
= -\ch^\beta_2(E) + i \alpha H. \ch_1^\beta(E) + \frac{\alpha^2}2 H^2\ch_0(E).
\end{equation}
Note that $Z_{\alpha, \beta}$ factors via the Chern character
\begin{equation} \label{eq:ch}
\ch \colon K(X) \to \Halg \cong \Z^3.
\end{equation}

We will first state Bridgeland's result constructing stability conditions on $\Db(X)$, and then explain its meaning.
\begin{Thm}[{\cite[Section 15]{Bridgeland:K3}}] \label{thm:stabconstr}
For $\alpha > 0, \beta \in \R$, the pair $\sigma_{\alpha, \beta} := \left(\Coh^\beta X, Z_{\alpha, \beta}\right)$
defines a stability condition on $D^b(X)$ satisfying the support property. 
Moreover, the map $\R_{>0} \times \R \to \Stab(X)$ is continuous.
\end{Thm}
We refer also to \cite{Emolo-Benjamin:lecture-notes} or \cite{wallcrossing-BrillNoether} for more details and a sketch of the proof. Up to an action of the universal cover
of $\GL_2^+(\R)$, the above theorem in fact describes an entire component of
$\Stab(X)$, but that fact will be irrelevant for us.

For our purposes, Theorem \ref{thm:stabconstr} makes two statements. First, consider
the slope function
\[ \nu_{\alpha, \beta} \colon \Coh^\beta X \to \R \cup \{+\infty\}, \quad \nu_{\alpha, \beta} (E) := \begin{cases}
\frac{-\Re Z_{\alpha, \beta}(E)}{\Im Z_{\alpha, \beta}(E)} & \text{if $\Im Z_{\alpha, \beta}(E) > 0$} \\
+\infty & \text{if $\Im Z_{\alpha, \beta}(E) = 0$.}
\end{cases}
\]
This defines a notion of slope-stability in $\Coh^\beta X$: every objects admits a Harder-Narasimhan filtration into slope-semistable objects; every slope-semistable objects admits a Jordan-H\"older filtration into slope-stable objects of the same slope. 

Moreover, that these satisfy the support property just follows from the fact that on an abelian surface, the Chern character $\vv$ of any stable object satisfies $\vv^2 \ge 0$; see Section \ref{subsect:modulispaces} for a discussion of this fact.

The second statement, the continuity, implies that this family satisfies wall-crossing as $\alpha, \beta$ vary. To give an efficient description of the walls, we change our viewpoint slightly. Observe that
up to the action of $\GL_2(\R)$ on $\C \cong \R^2$, a central charge 
$Z \colon K(X) \to \C$ that factors via the Chern character \eqref{eq:ch}
is uniquely determined by its kernel $\Ker Z \subset \HalgR \cong \R^3$. 
The Mukai pairing:
\begin{align*}
& \langle(v_0,v_1,v_2),(w_0,w_1,w_2)\rangle 
:= v_1w_1-v_0w_2-v_2w_0
\end{align*}
equips $\HalgR$ with a quadratic form $Q$ of signature $(2, 1)$.
The definition of $\Zab$ in equation \eqref{eqn:Zab} identifies the upper half plane $\R_{>0} \times \R$
with the projectivisation of the negative cone of $Q$, via the correspondence $(\alpha, \beta) \mapsto \Ker \Zab$.

\begin{Prop}[{\cite[Proposition 9.3]{Bridgeland:K3}}] \label{prop:walls}
Let $\vv \in \Halg$ be a primitive algebraic cohomology class. Then there exists a collection of two-dimensional linear subspaces
$\cW_v^i \subset \HalgR \cong \R^3$ containing $\vv$ with the following properties:
\begin{itemize}
\item there exists a strictly $\sigma_{\alpha, \beta}$-semistable object with Chern character $\vv$ if and only if $\Ker \Zab$ is contained in one of the subspaces $\cW_v^i$; 
\item the intersection of these subspaces with the negative cone of $Q$ is locally finite; and
\item as $\alpha, \beta$ vary within a chamber, i.e.~without crossing a wall, (semi-)stability with respect to
$\sigma_{\alpha, \beta}$ is unchanged for objects of Chern character $\vv$.
\end{itemize}
In the case where $\vv^2 = 0$, the set of walls is empty.
\end{Prop}
See Figure~\ref{fig:wallspicture} for a picture;
we also refer to \cite[Theorem 3.1]{Maciocia:walls} for a proof of the fact that the walls are nested when viewed as semi-circles in the upper half plane.

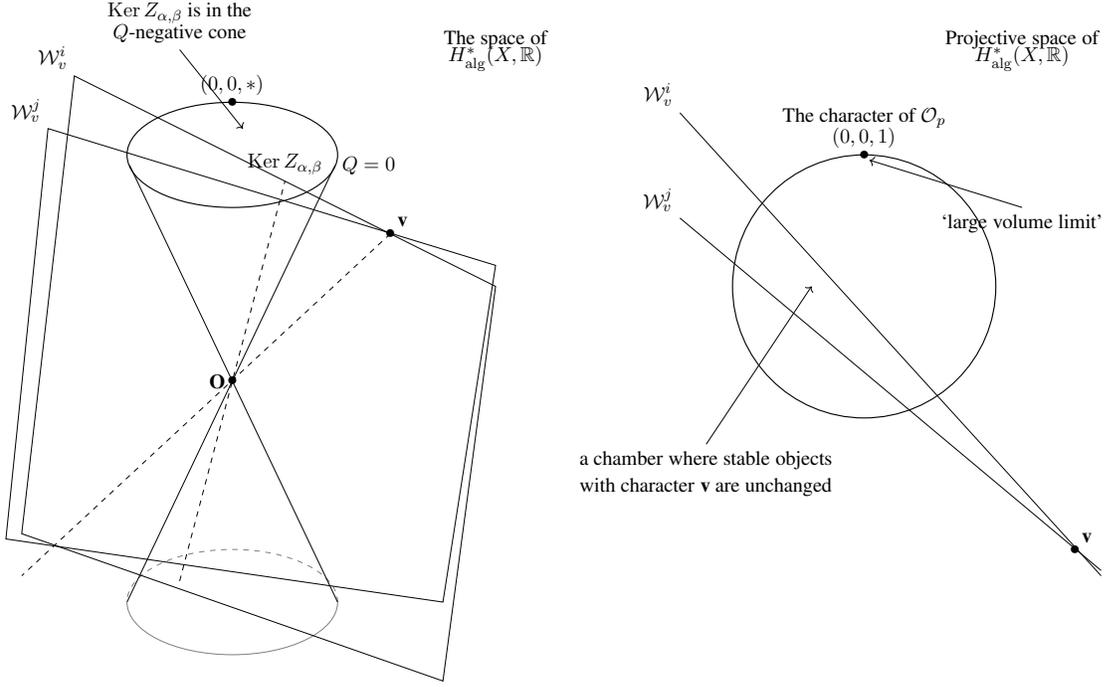
\begin{figure}[htb]
\begin{center}
\scalebox{0.7}{
\begin{tikzpicture}
       \coordinate (O) at (0,0.2);
	\draw[dashed,color=gray] (2,-4) arc (0:180:2 and 1);
	\draw[color=gray] (2,-4) arc (0:-180:2 and 1);
		\draw[semithick] (0,4.5) ellipse (2 and 1);
		\draw (1.95,4.3) node[right] {$Q=0$}--(-2,-4);
		\draw (-1.95,4.3) --(2,-4); \draw (0,5.5) node {$\bullet$} node[above] {$(0,0,*)$};
	\draw (5,6.7) node{The space of} node [below]{ $\HalgR$};
	
	\draw[->] (-1,6.5) node[above]{$Q$-negative cone} -- (0.2,5);
	\draw (-1,7.2) node{$\Ker Z_{\alpha,\beta}$ is in the};
	\draw (O) node{$\bullet$} node[left] {\textbf{O}};
	\draw (5,2) -- (-3,6) node[above left] {$\mathcal W_v^i$} -- (-4,-2.7) -- (4,-5.5)--cycle;
	\draw (5,2.4) -- (-3.5,5) node[above left] {$\mathcal W_v^j$}-- (-4.3,-2.8) -- (4,-4)--cycle;
	\draw (3,3) node {$\bullet$}  node[above right]{\textbf{v}};
	\draw[dashed] (3,3) -- (-4,-3.5);
	\draw[dashed] 
    (-1,-3.6)--(O) -- (1,4) node[above]{$\Ker Z_{\alpha,\beta}$};
	\draw (15,6.7) node{Projective space of} node [below]{ $\HalgR$};
	\draw[semithick] (12,2) ellipse (2.5 and 2.5);
	\draw[->] (15,3.5) node[below]{`large volume limit'} --(12.1,4.4);
    \draw (12,4.5) node {$\bullet$} node[above] {$(0,0,1)$};
	\draw (12,5.2) node {The character of $\cO_p$};
    
	\draw (16,-3) node {$\bullet$}  node[above right]{\textbf{v}};
	\draw (16.5,-3.5) -- (8.5,5.3) node[above left] {$\mathcal W_v^i$};
	\draw (16.5,-3.4) -- (8.5,3.3) node[above left] {$\mathcal W_v^j$};
	\draw[->]  (9,-1) node[below] {a chamber where stable objects} -- (11,2);
	\draw (9,-1.5) node[below] { with character \textbf{v} are unchanged};
\end{tikzpicture}
}
\end{center}
\caption{Describing walls via $\Ker Z_{\alpha, \beta} \subset H^*(X, \R)$} \label{fig:wallspicture}
\end{figure}

The final claim, for the case $\vv^2 = 0$, is a general fact for stability conditions satisfying the support property with respect to a given quadratic form, see e.g.~\cite[Proposition A.8]{BMS:stabCY3s}.

\subsection{Moduli spaces and large-volume limit}
\label{subsect:modulispaces}
Recall that on an abelian surface, there are no rigid objects (see \cite[Lemma 15.1]{Bridgeland:K3} for a short proof). Therefore, if the moduli space
$M_{\sab}(\vv)$ of $\sab$-stable objects of Chern character $\vv \in H^*(X, \Z)$ is non-empty, then Hirzebruch-Riemann-Roch shows $\vv^2 \ge 0$. It turns out that this necessary condition is also sufficient.
\begin{Thm}[{\cite{Yoshioka:cones, MYY2}}]
Let $v \in \Halg$ be a primitive algebraic cohomology class. If
$\alpha, \beta$ are generic, then the moduli space $M_{\sigma_{\alpha, \beta}}(\vv)$ of 
semistable objects of Chern character $\vv$ is non-empty if and only if
$\vv^2 \ge 0$. In this case it is a smooth holomorphic symplectic variety\footnote{Note that it is not an \emph{irreducible} holomorphic symplectic variety when $\vv^2 > 0$; instead it is deformation equivalent to $\Hilb^{\vv^2/2}(X) \times X$.} and has dimension $\vv^2 + 2$.
\label{thm:mukai}
\end{Thm}
\begin{proof}
In the case of $\vv^2 > 0$, this is part of the statement of \cite[Theorem 1.13]{Yoshioka:cones} and \cite[Proposition 5.16]{MYY2}. (The statement ``deformation-equivalent to \dots'' in \cite[Proposition 5.16]{MYY2} in particular includes the non-emptiness.)

As for $\vv^2 = 0$, in this case there are no walls as observed in Proposition \ref{prop:walls}. Therefore, we may assume that $\sigma$ is near the large-volume limit, in which case it is well-known that the Gieseker-moduli space $M_\sigma(\vv) = M_H(\vv)$ is an abelian surface derived equivalent to $X$.
\end{proof}

Now fix such a primitive cohomology class $\vv$ of positive rank. If $\beta$ remains fixed,
and $\alpha \leadsto +\infty$, then the phase of $Z_{\alpha, \beta}(E)$ is asymptotically governed by the slope $\mu_H(E)$. The following result will refine
that observation.

We write $M_H(\vv)$ for the moduli space of Gieseker-stable sheaves with respect to the polarization $H$. Given $E\in \Db(X)$, let $E^\vee = \RlHom(E, \cO_X)$ be its derived dual, and write $\vv^\vee$ for the class dual to $\vv$. 

\begin{Prop}[{\cite[Proposition 14.2]{Bridgeland:K3}}] \label{prop:largevolume}
If $\beta < \mu_H(\vv)$ and if $\alpha$ is sufficiently big,
then $M_{\sab}(\vv) = M_H(\vv)$: an object $E \in \Coh^\beta(X)$ of Chern character $\vv$ is $\sab$-stable if an only if it is a Gieseker-stable sheaf.

If $\beta > \mu_H(\vv)$ and $\alpha$ is sufficiently big, then $M_{\sab}(-\vv) = \D(M_H(\vv^\vee))[1]$: an object $E \in \Coh^\beta(X)$ of Chern character $-\vv$ is $\sab$-stable if and only if
it is the shift $E = F^\vee[1]$ of the derived dual of a Gieseker-stable sheaf $F$ of class $\vv^\vee$.
\end{Prop}
In other words, there is a Gieseker-chamber in which Bridgeland stability and Gieseker stability coincide for objects of Chern character $\vv$; and similarly there is a dual Gieseker chamber in which they coincide up to taking derived duals.

We also make the following observation for non-primitive classes of square zero:

\begin{Prop} \label{prop:mv0square0}
Let $\vv = m \vv_0$ with $m > 1$ and $\vv_0^2 = 0$. Then every $\sigma$-semistable object of class $\vv$ is strictly semistable, and all its Jordan-H\"older factors are of class $\vv_0$. 
\end{Prop}
In particular, such an object has a stable quotient and a stable subobject of the same phase, each of class $\vv_0$.
\begin{proof}
By the same argument as in the proof of \cite[Proposition A.8]{BMS:stabCY3s}, the set of semistable objects of class $\vv$ is constant as $\sigma$ varies. Hence we can assume $\sigma$ is in the Gieseker-chamber for $\vv_0$ and $\vv$. The moduli space $M_\sigma(\vv_0)$ is an abelian surface derived equivalent to $X$ (up to a Brauer twist). The associated Mukai transform sends objects
in $M_{\sigma}(\vv_0)$ to skyscraper sheaves of points, and objects
in $M_{\sigma}(\vv)$ to 0-dimensional torsion sheaves of length $m$. This proves the claim.
\end{proof}

\def\hatX{\widehat{X}}

Finally, we need the construction of associated Kummer varieties. Let $\widehat X := \Pic^0(X)$ be the dual abelian variety, and let $\cP$ be the Poincar\'e line bundle on $X \times \widehat{X}$. Let $\Phi^\cP \colon \Db(X) \to \Db(\hatX)$ be the associated Fourier-Mukai transform. Given $E \in \Db(X)$, let
$\det(E)$ be the associated determined line bundle. Now consider a moduli space $M_\sigma(\vv)$ with chosen basepoint $E_0$. We obtain a map
\begin{eqnarray}
& \underline{\det}  \colon   & M_\sigma(\vv) \to \hatX, \quad E \mapsto \det(E) \otimes \det(E_0)^\vee, \quad \text{and, dually,} \label{eq:defdet} \\
& \widehat{\underline{\det}} \colon  & M_\sigma(\vv) \to X, \quad E \mapsto \det(\Phi^\cP E) \otimes \det(\Phi^\cP E_0)^\vee.
\end{eqnarray}

\begin{Prop}[Yoshioka] \label{prop:albanese}
Assume that $\vv^2 \ge 2$ is primitive. Then the Albanese map of $M_\sigma(\vv)$ is given by 
\[ \underline{\det} \times \widehat{\underline{\det}} \colon M_\sigma(\vv) \to X \times \hatX, \]
and it has connected fibers. 

If $\vv^2 = 0$ and with $\vv$ primitive, then $M_\sigma(\vv)$ is an abelian surface. If moreover the rank $\rk(\vv)$ is positive, then 
$\underline{\det} \colon M_\sigma(\vv) \to \hat X$ is a finite map of degree $\rk(\vv)^2$. 
\end{Prop}
\begin{proof}
For $\vv^2 \ge 6$, this is part of the statement of \cite[Proposition 5.16, (2a)]{MYY2}. For $\vv^2 = 4$ and $\vv^2 = 2$ this is proven for moduli spaces of Gieseker-stable sheaves in
\cite[Section 3]{Yoshioka:albanese} and \cite[Corollary 4.3]{Yoshioka:some_notes}. This property remains preserved under wall-crossing, and thus holds for any $M_{\sigma}(\vv)$, with the same proof as the one given in \cite{MYY2} for $\vv^2 \ge 6$.

In case $\vv^2 = 0$, then $M_\sigma(\vv)$ is (up to shift) a moduli space of semihomogeneous vector bundles. In particular, for $E \in M_\sigma(\vv)$ there is a map $\Pi \colon \hatX \to M_\sigma(\vv)$, $L \mapsto E \otimes L$ that is surjective \cite[Proposition 6.10]{Mukai:semi-homogeneous} of degree $\rk(\vv)^2$ \cite[Proposition 7.1]{Mukai:semi-homogeneous}. Since the composition $\underline{\det} \circ \Pi$ is multiplication by $\rk(\vv)$, and thus of degree
$\rk(\vv)^4$, the claim follows. 
\end{proof}
If in fact $\vv^2 \ge 4$, then the fibers of the Albanese map are irreducible holomorphic symplectic varieties (or K3 surfaces) of dimension $\vv^2 - 2$, but we will not need that fact.

\section{The Gieseker-wall}

The Gieseker chamber described by Proposition \ref{prop:largevolume} typically has a wall corresponding to the Gieseker-Uhlenbeck contraction, corresponding to the subspace spanned by $\vv$ and $(0, 0, 1)$.
In this section, we will describe its second wall for 
classes $\vv$ with $c_1(\vv) = H$ and $\chi(\vv) < 0$. 

\begin{Prop} \label{prop:gwall}
Let $\vv = (r, H, \chi) \in \Halg$ with $r \ge 0$.
For any $\alpha >0$ and $\beta = 0$, we have
$M_{\sigma_{\alpha, 0}}(\vv) = M_H(\vv)$. Similarly, all line bundles $L \in \Pic^0(X)$ of degree 0 are $\sigma_{\alpha, 0}$-stable.
\end{Prop}
\begin{proof}
Observe that 
\begin{eqnarray*}
\Im Z_{\alpha, 0}(E)
	& \in & \Z_{\ge 0} \cdot \alpha H^2 \quad \text{for all $E \in \Coh^\beta X$, and} \\
\Im Z_{\alpha, 0}(\vv) & = & \alpha H^2.
\end{eqnarray*}
Therefore, an object $E \in \Coh^0 X$ with $\ch(E) = \vv$ can never by strictly $\sigma_{\alpha, 0}$-semistable: its Jordan-H\"older factors $E_i$ would satisfy
$\Im Z_{\alpha, 0}(E_i) \in (0, \alpha H^2)$ in contradiction to the first equation. Combined with Propositions \ref{prop:walls} and \ref{prop:largevolume}, this proves the first claim.

The stability of $L \in \Pic^0(X)$ immediately follows from the last claim of Proposition \ref{prop:walls}, again combined with Proposition \ref{prop:largevolume}.
\end{proof}

\begin{figure}[htb]
\begin{center}
\tikzset{%
    add/.style args={#1 and #2}{
        to path={%
 ($(\tikztostart)!-#1!(\tikztotarget)$)--($(\tikztotarget)!-#2!(\tikztostart)$)%
  \tikztonodes},add/.default={.2 and .2}}
}
\scalebox{0.7}{
\begin{tikzpicture}      
         \coordinate (O) at (0,0);
       \draw (O) node {$\bullet$} node[below] {$\cO_X$};
       
       \coordinate (P) at (0,6);
       \draw (P) node {$\bullet$} node[above right] {`large volume limit'};
       
       \coordinate (V) at (1,-0.7);
       \draw (V) node {$\bullet$} node[below] {$\vv$};
       
       \coordinate (B) at (0,1.4);
       \draw (B) node {$\bullet$} ;
       \draw[->] (4,2.5) node[right] {$M_H(\vv)=M_{\sigma_{\alpha,0}}(\vv)$} to (0.1,1.5);
       
	\draw[semithick] (0,3) ellipse (3 and 3);
	
        \draw [add= 1.5 and 9] (V) to (O) node[above] {The wall $W_\chi$};
        \draw [add= -4.0 and 3.1] (V) to (O) node {$\bullet$} node [below left] {$(-R,H,\chi)$};
        \draw [add= -1.2 and 1.5] (V) to (O) node {$\bullet$} node [below] {$\sigma_0$};
	\coordinate (SP) at (-1.2, 1.2);
	\draw (SP) node {$\bullet$} node[above] {$\sigma_+$};
        
        \draw [add= 0 and 0.3] (V) to (P) node[above] {Gieseker-Uhlenbeck wall};       
        \draw [add= 0 and 1.1, dotted] (V) to (B) node {$\bullet$} node[left] {$M_H(\vv)$} node[below left] {$=M_{\sigma}(\vv)$};
       \draw [add= -2.1 and 3.4, dotted] (V) to (B) node[above] {A ray inside the Gieseker-chamber
for $\vv$};
        
        \draw [dashed] (O) to (P);
         \draw [->] (4,4) node [above right]{the line of } node[below right]{$\alpha>0, \beta=0$} to (0.1,3);
         \draw [add= 3 and 7.6, dashed] (P) to (-1,6) node {$\bullet$} node [below left] {$(0,H,\chi)$};
         \draw [add= -8.6 and 8.6, dashed] (P) to (-1,6);
\end{tikzpicture}
}
\end{center}
\caption{The space $M_{\sigma}(\vv)$ is unchanged in the chamber bounded by walls of $\overline{\vv\cO_p}$ and $\overline{\vv\cO_X}$.} \label{fig:firstwall}
\end{figure}
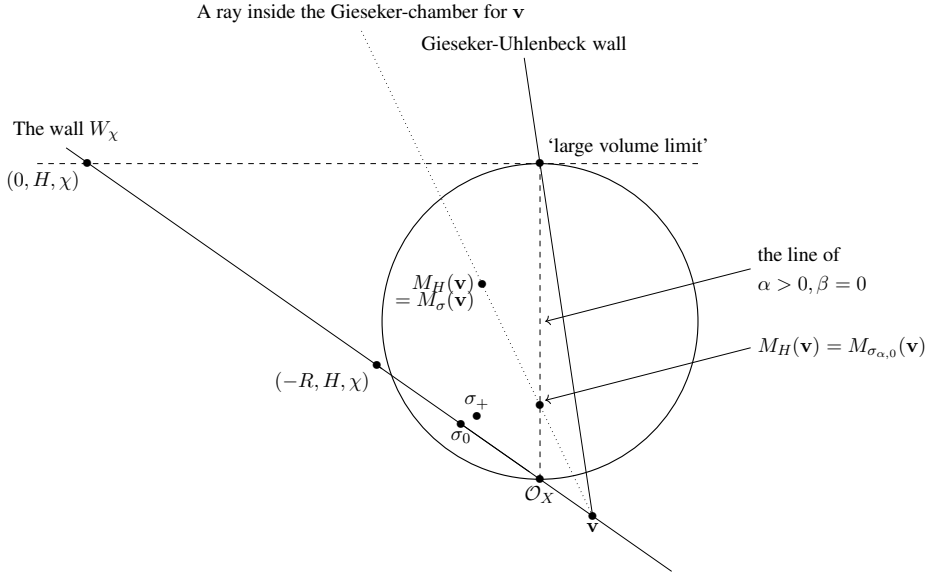

The line connecting $\vv$ with $\Ker Z_{0, 0}= \R \cdot (1, 0, 0)$ is therefore the first possibility for a second wall of the Gieseker-chamber for $\vv$. For $\chi < 0$, this wall does exist and will be described in the following.

\begin{Def} \label{def:wk}
\begin{enumerate}
\item
Given $\chi<0$, let $R$ be the number
\[
R := \max\left\{r\in \mathbb Z_{\geq 0}\; \middle|\; (-r,H,\chi)^2=H^2+2r\chi\geq 0\right\}
= \left\lfloor \frac{H^2}{-2\chi} \right\rfloor .
\]
We denote $w_0$ by the character $(-R,H,\chi)\in \Halg$, and $w_k=w_0+(k,0,0)$ for $k\in \Z_{\geq 0}$. 
\item
Let $\cW_\chi \subset \HalgR \cong \R^3$ be the  two-dimensional subspace spanned by $(1,0,0)$ and $w_0$ (and containing all $w_k$). Denote by $\sigma_0=\left(\Coh^\beta X, Z_{\alpha, \beta}\right)$  a stability condition on the wall corresponding to $\cW_\chi$: it is given as $\sigma_{\alpha_0, \beta_0}$ as in Theorem \ref{thm:stabconstr} for any $\alpha_0, \beta_0$ with $\Ker Z_{\alpha_0, \beta_0} \subset \cW_\chi$.
\item Finally, let $\cP_0$ denote the category of $\sigma_0$-semistable objects in $\Coh^{\beta_0} X$ of the same slope $\nu_0$ as objects of character $w_k$ (for any $k \ge 0$). The category $\cP_0$ does not depend on the choice of $\sigma_0$ on $\cW_\chi$.
\end{enumerate}
\end{Def}
Given an object $E \in \cP_0$, its Chern character is a linear combination 
\begin{equation} \label{eq:chElincomb}
\ch(E) = a(1, 0, 0) + b w_0.\end{equation}
Since $H$ generates the Picard group, $b$ has to be integral, and therefore $a$ as well.
Moreover, $Z_{\alpha_0, \beta_0}(E)$ has to lie on the same ray as $Z_{\alpha_0, \beta_0}(w_k)$; combined with $\ch(E)^2 \ge 0$ this is only possible if $b \ge 0$. This leads to the following observation:

\begin{Lem}
Let  $E \in \cP_0$ be a $\sigma_0$-semistable object with $\ch(E) = w_k$. Then for each of its Jordan-H\"older factor $E_i$, the Chern character  is either given by $\ch(E_i) = w_{k'}$ for some $0 \le k' \leq k$, or by $\ch(E_i) = (1,0,0)$. Moreover, $\cW_\chi$ is a wall for the Chern character $w_k$ for all $k > 0$.
\label{Owall}
\end{Lem}

\begin{proof}
Since $w_k - \ch(E_i)$ is the sum of the characters of the remaining Jordan-H\"older factors, we must, in addition to the observations of the previous paragraph have $b \le 1$.

If $b=0$, i.e.~$\ch(E_i)=(a,0,0)$, then Proposition \ref{prop:mv0square0} shows $\ch(E_i) = (1, 0, 0)$. If $b=1$, we have $\ch(E_i)=w_{k-a}$.


It remains to prove that $\cW_\chi$ is a wall for $w_k$ for $k>0$. By Theorem \ref{thm:mukai} and the definition of $w_k$ and $R$, there exists a stable object $E'$ with Chern character $w_{k-1}$. Then $E'\oplus \cO_X$ is a strictly $\sigma_0$-semistable object, and so $\cW_\chi$ describes a wall for the Chern character $w_k$.
\end{proof}

Fix $K \ge 0$.
Let $\sigma_+$ be a geometric stability condition, i.e.~one given as in Theorem \ref{thm:stabconstr}, sufficiently nearby the wall corresponding to $\cW_\chi$, on the side where the associated slope function satisfies $\nu_+(\cO)<\nu_+(w_k)$ for $k\geq 0$. It is immediate from the local finiteness of walls (see Proposition \ref{prop:walls}) that $\sigma_+$ is not separated by a wall from $\cW_\chi$ for any $w_k$ and $0 \le k \le K$; in particular,
$M^{s}_{\sigma_+}(w_k) = M_{\sigma_+}(w_k) \subset M^{ss}_{\sigma_0}(w_k)$. By Proposition \ref{prop:gwall}, the stability condition $\sigma_+$ is in fact in the (dual of the) Gieseker-chamber of $w_k$. Moreover:

\begin{Prop}[{\cite[Proposition 5.1]{BM:walls}}] \label{prop:HNnearby}
Let $E$ be $\sigma_0$-semistable and of Chern character $w_k$ for $0 \le k \le K$. Let $E_i$ be any of its Harder-Narasimhan filtration factors with respect to $\sigma_+$. Then
$E_i$ is $\sigma_0$-semistable, and its Chern character is contained in $\cW_\chi$.
Moreover, the Jordan-H\"older factors of $E_i$ are a subset of the Jordan-H\"older factors of $E$.
\end{Prop}

\begin{Cor}
Let  $E \in \Coh^{\beta_0}$ be a $\sigma_0$-semistable object with character $w_k$. Then it is $\sigma_+$-stable if and only if for every $L \in \Pic^0(X)$ we have $\Hom(E,L)=0$. 
\label{cor:ston}
\end{Cor}

\begin{proof}
Since $\nu_+(L) = \nu_+(\cO) < \nu_+(w_k)$, and since $L$ is $\sigma_+$-stable by Proposition \ref{prop:walls}, the condition is clearly necessary. 

Conversely, assume that $E$ is $\sigma_0$-semistable but unstable with respect to $\sigma_＋$. By
Proposition \ref{prop:HNnearby}, there is a destabilizing quotient 
$E \onto E'$ such that $E'$ is $\sigma_0$-semistable with $\ch(E') \in \cW_\chi$, and such that the Jordan-H\"older factors of $E'$ are a subset of those of $E$. 
By Lemma \ref{Owall}, the character of $E'$ is either equal to $w_{k'}$ for $k' < k$, or equal to $(r, 0, 0)$. Since $\nu_+(E)\geq \nu_+(E')$, the former case is impossible, and so $\ch(E') = (r, 0, 0)$. In light of Proposition \ref{prop:mv0square0}, we can in fact assume $r = 1$, and so $E' = L$ for some $L \in \Pic^0(X)$.
\end{proof}

The following lemma is well-known.
\begin{Lem} \label{lem:reduceS}
Let $E$ be an object that is $\sigma_0$-semistable, and let $S$ be one of its JH factors. Then there exists a unique short exact sequence
\[ T \into E \onto E' \]
such that $T$ is $\sigma_0$-semistable with all JH factors isomorphic to $S$, and such that $\Hom(S, E') = 0$.
\label{dec}
\end{Lem}
\begin{proof}
The existence of $E'$ is proved by induction on the length of the Jordan-H\"older filtration of $E$.
If $\tilde T \into E \onto \tilde E'$ is another such short exact sequence, then $\Hom(T, \tilde E') = 0 =
\Hom(\tilde T, E')$, and a simple diagram-chase proves the uniqueness.
\end{proof}

We will use this in the following context.

\def\ere{E_{\mathrm{red}}}
\begin{Def}
Let $E$ be an object an object that is $\sigma_0$-semistable. We apply Lemma \ref{lem:reduceS} with $S = \cO_X$ and define $E_{\mathrm{red}}$  as the unique object with a short exact sequence
\begin{equation}
T \into E \onto E_{\mathrm{red}},
\label{er}
\end{equation}
such that $T$ is an iterated self-extension of $\cO_X$,
and $\Hom(\cO_X, \ere) = 0$.
\label{erd}
\end{Def}

\section{Main proof}

We continue to use the notation from the previous section, in particular see Definition
\ref{def:wk}, and we continue to assume $\chi < 0$.  The goal of this section is Theorem \ref{thm:maintechnical}, on non-emptiness and dimension of the loci
\[\mhk: =\left\{E\in M^{s}_{\sigma_+}(w_k)\;\middle| \; \hom(\cO_X,E)=h\right\}.\]

The strategy is to control the existence of objects in $\mhk$, and to construct them, via the two short exact
sequences appearing in Lemma \ref{lem:2ses}, respectively. First we observe the following:

\begin{Lem} \label{lem:extdim-new}
If $E \in \mhk$, then $\dim \Ext^1(E, \cO_X) = h - \chi$.
\end{Lem}
\begin{proof}
This follows from Serre duality, $\chi(E) = \chi$, and $\Hom(E, \cO_X) = 0$ by stability.
\end{proof}

\begin{Lem} \label{lem:2ses}
Let $E$ be an object in $M^h_k$.
\begin{enumerate}
\item \label{enum:sesW}
Given a subspace $W \subset H^0(E)$, consider the short exact sequence in $\cP_0$ given as
\begin{equation} \label{eq:sesW}
 \cO_X \otimes W \into E \onto E'.
\end{equation}
Then $E'$ is $\sigma_+$-stable with $h^0(E') \ge \dim W + \chi$.
\item \label{enum:sesV}
Conversely, let $V \subset \Ext^1(E, \cO_X)$, and consider the natural short
exact sequence 
\begin{equation} \label{eq:sesV}
\cO_X \otimes V^\vee \into \widetilde{E} \onto E.
\end{equation}
Then $\widetilde{E}$ is $\sigma_+$-stable with
$h^0(\widetilde E) \ge \dim V$.
\end{enumerate}
\label{lem:ext}
\end{Lem}
\begin{proof}
Since $E'$ and $\widetilde{E}$ are $\sigma_0$-semistable, in light of Corollary \ref{cor:ston} we
need to verify that they do not admit $L \in \Pic^0(X)$ as a quotient in $\cP_0$. This is immediate
for $E'$. For $\widetilde E$, we consider the long exact sequence associated via $\Hom(\blank, L)$:
\[ 0 = \Hom(E, L) \to \Hom(\widetilde E, L) \into \Hom(\cO_X, L) \otimes V \to \Hom(E, L[1]). \]
For $L \neq \cO_X$, the vanishing is immediate. For $L = \cO_X$, it follows by our choice of
$V$.

The bound for $h^0(\widetilde{E})$ is immediate, as $W$ is a subspace of $H^0(\widetilde{E})$ and the quotient 
injects into $H^0(E)$. As for $E'$, observe that $H^2(E') = \Hom(E', \cO_X)^\vee = 0$ and, due to
the existence of the non-trivial extension $E$, also $\dim H^1(E') = \Ext^1(E', \cO_X) \ge \dim W$.
Combined with $h^0(E') = \chi + h^1(E')$ this shows the remaining claim.
\end{proof}

They key difficulty is that for both short exact sequences above, we
only obtain a bound for $h^0$ of the corresponding object, and that $h^0(E')$ may be non-zero even
for $W = H^0(E)$. (In contrast, in the case of K3 surface,
these dimensions are directly determined due to $h^1(\cO_X) = 0$.) As we have explained previously, we need to proceed by induction. Moreover, in order to control
the dimension of $H^0$ precisely for \emph{some} of these extensions, it turns out that we have to
prove more precise statements in our induction claim. These more precise claims are based on
Lemma \ref{lem:reduceS}, by involving the class of $E_{\mathrm{red}}$.

\def\kred{k_{\mathrm{red}}}
\begin{Cor}
Let $E$ be an object in $\mhk$, then $\ere$ given in Definition \ref{erd}  is $\sigma_+$-stable. In
particular, $\ere \in M^0_{\kred}$ for some $k_{\mathrm{red}}\in\Z_{\geq 0}$. 
\label{ered}
\end{Cor}

\begin{proof}
This follows from the inductive construction of $\ere$ in Lemma \ref{lem:reduceS}, and Lemma \ref{lem:2ses}.
\end{proof}

\def\mhkkred{M_{k,k_{\mathrm{red}}}^h}
\begin{Def}
We define  a locally closed subset $\mhkkred$  of $\mhk$ by
\[\mhkkred:=\left\{E\in \mhk\; \middle| \; \ere\in M^0_{\kred}\right\}.\]
\label{def:mkt}
\end{Def}

Our induction process will be controlled by the following piece-wise linear function:

\begin{Def}\label{def:f}
Let $\Delta_{KLM}\colon \R_{\geq 0} \rightarrow \R$ be a function inductively defined as follows:
\begin{itemize}
\item $\Delta_{KLM}(t):=t$, when $0\leq t\leq 1$;
\item $\Delta_{KLM}(t):=\Delta_{KLM}(t-1)+t$ for $t \ge 1$.
\end{itemize}
\end{Def}
It is a continuous piece-wise linear function with $\Delta_{KLM}(n) = \frac{n(n+1)}2$ for $n \in \Z_{\geq 0}$; explicitly,
\begin{equation} \label{eq:Deltaexplicit}
\Delta_{KLM}(t) = \left(t-\frac{\floor*{t}}{2}\right)\left(\floor*{t}+1\right)
\quad \text{for all $t \in \R_{\geq 0}$.}
\end{equation}

\begin{Prop}
The space $\mhkkred$ is empty when
\begin{equation}\frac{k-k_{\mathrm{red}}}{-\chi}< \Delta_{KLM}\left(\frac{h}{-\chi}\right). \label{eq:klm1} \end{equation}
\label{mhtkn}
\end{Prop}
Before the proof, let us use the statement of the Proposition \ref{mhtkn} in order to illustrate the purpose
of the function $\Delta_{KLM}$. Consider the short exact sequence \eqref{eq:sesW} for $W = H^0(E)$,
and assume that $E \in \mhkkred$ for a tuple $k, \kred, h$ that satisfies \eqref{eq:klm1}. Then $E'
\in M^{h'}_{k', \kred}$ for $k' = k - h$ and $h' \ge h + \chi$. Then the tuple $k', \kred, h'$ also
satisfies \eqref{eq:klm1} precisely because of the functional equation satisfied by $\Delta_{KLM}$.

\begin{proof}

When $k\leq  k_{\mathrm{red}}-\chi$, suppose there is an object $E\in M^{h}_{k,k_{\mathrm{red}}}$
such that inequality (\ref{eq:klm1}) holds, i.e.~$h\geq k-k_{\mathrm{red}}+1\geq 1$. Consider the
quotient $E'$ defined by
\begin{equation}
\cO_X^{\oplus h}\cong \cO_X\otimes \Hom(\cO_X,E) \into  E \onto E'.
\label{eq13}
\end{equation}
Note that, by Definition \ref{erd},  $\ere \cong (E')_{\mathrm{red}}$. But since $E'\in M_{k-h}$, we
have $(E')_{\mathrm{red}}\in M_{t'}=M_{\kred}$ for $t'\leq k-h\leq k_{\mathrm{red}}-1$, this leads the contradiction.

 When $k>k_{\mathrm{red}}-\chi$, we proceed by induction on $k$. Suppose  there is an object $E\in
M^{h}_{k,k_{\mathrm{red}}}$ such that the  (\ref{eq:klm1}) holds. In particular, we have $h>-\chi$.
We again consider $E'$ fitting into the short exact sequence (\ref{eq13}). By Lemma \ref{lem:2ses}, $E'$ is $\sigma_+$-stable
with $h^0(E') \ge h + \chi$.

As before, $\ere\simeq (E')_{\mathrm{red}}$, therefore, $M_{k-h,k_{\mathrm{red}}}^{h'}\neq \emptyset$ for
some $h' \ge h + \chi$. By induction on $k$,
\[
 \frac{k-h-k_{\mathrm{red}}}{-\chi}\geq \Delta_{KLM}\left(\frac{h'}{-\chi}\right) 
\ge  \Delta_{KLM}\left(\frac{h+\chi}{-\chi}\right) = \Delta_{KLM}\left(\frac{h}{-\chi}\right) -\frac{h}{-\chi}.
\]
This contradicts the assumption on $k$. 
\end{proof}

\begin{Def} 
We write $d(k, h)$ for the expected dimension of $\mhk$, which is given as
\[d(k,h):= w_0^2-2k\chi+2+h\chi-h^2 = w_k^2 + 2 + h\chi - h^2. \]
\end{Def} 

\begin{Lem}
Every irreducible component of $M^h_k$ has dimension at least $d(k,h)$.
\label{lem:dkh}
\end{Lem}
\begin{proof}
This could be proved similarly to the case of  line bundles on curves, see e.g.~\cite[Section
IV.3]{ACGH:1}; however, to treat the cases $k \ge R$ (in which case $M_{\sigma_+}(w_k)$ parameterizes
stable sheaves) and $k < R$ (in which case it instead parameterizes derived duals of stable sheaves)
simultaneously, we present here a derived category version of the classical argument.

Consider any family $\cE \in D^b(S \times X)$ of $\sigma^+$-semistable objects of class $w_k$ over a
scheme $S$ of finite type over $\C$;
this means that the derived restriction $\cE_s := \cE|_{\{s\} \times
X} \in D^b(X)$ is an $\sigma^+$-stable object of class $w_k$ for every closed point $s$. We will
prove that the locus 
\[ S_h := \stv{s \in S}{h^0(\cE_s) = h} \subset S \]
has codimension at most $-h\chi + h^2$.

Let $\pi_S \colon S \times X \to S$ denote the projection, and consider the derived push-forward $\cF:= (\pi_S)_* \cE$. By derived base change (see e.g.~\cite[Tag~08IB]{stacks-project}), we have $H^i (\cF \otimes^{\mathbf{L}} \cO_s) = H^i(\cE_s) = 0$ for $i \neq 0, 1$ (note that $H^2(\cE_s)=\Hom(\cE_s,\cO_X)^\vee=0$ since $\cE_s$ is $\sigma_+$-stable). By \cite[Proposition
5.4]{Bridgeland-Maciocia:K3Fibrations}, it follows that $\cF$ is quasi-isomorphic to a 2-term complex of vector bundles
$\cF_0 \to \cF_1$,
and of rank $\chi = \rk \cF_0 - \rk \cF_1$. Then $S_h$ is the locus where the rank of the
differential is given by $\rk \cF_1 - h$; this has codimension at most $h(h -\chi)$ as claimed.
\end{proof}

\begin{Rem}
The proof evidently applies in much bigger generality: $X$ could be an arbitrary scheme of finite
type of the base field; we only need to assume that $\cE \in D^b(S \times X)$ has the property
that for all $s \in S$, the restriction $\cE_s$ is in  $D^b(X)$, has compact support, and cohomology
in at most two degrees.
The classical proof instead constructs the two-term complex $\cF_0 \to \cF_1$ explicitly.
\end{Rem}

Let $D$ be the remainder of division of $h$ by $-\chi$, with
$0 \le D < -\chi$.
\begin{Prop} \label{prop:nonemptyMHkkred}
The moduli space $\mhkkred$ is non-empty when $k\geq k_{\mathrm{red}}\geq 0$ and $h$ is the maximum integer such that 
\begin{equation}
\frac{k-k_{\mathrm{red}}}{-\chi} \geq \Delta_{KLM}\left(\frac{h}{-\chi}\right).\label{thk}
\end{equation}
When equality holds,  $\mhkkred$ is a bundle  over $M^0_{\kred}$ of Grassmannians of $D$-dimensional subspaces in a $(-\chi)$-dimensional vector space, and of total dimension $d(k,h)$.
\label{mhtk}
\end{Prop}
\begin{proof}
When $k=k_{\mathrm{red}}$, since $w_{\kred}^2\geq 0$, the moduli space
$M_{\sigma_+}(w_{\kred})$ is non-empty by Theorem \ref{thm:mukai}. Since any $E \in
M_{\sigma_+}(w_{\kred})$ has only finitely many Jordan-H\"older factors with respect to $\sigma_0$,
we can use the action of $\Pic^0(X)$ to find an object that does not have $\cO_X$ as a factor, and
in particular satisfies $\Hom(\cO_X, E) = 0$. 
Therefore, $M_{\kred,\kred}^0$ is non-empty. Moreover, it is an open dense subset of
$M^s_{\sigma_+}(w_{\kred})$, and thus of expected dimension $d(k, 0)$.


Now consider the case $0<k-\kred < -\chi$. Then $h = k - \kred$ satisfies equality in
\eqref{thk}.
Let $E$ be an object
in $M^0_{k_{\mathrm{red}},k_{\mathrm{red}}}$. By Lemma \ref{lem:extdim-new}, $\Ext^1(E, \cO_X)$ has
dimension $-\chi$, and thus we can choose a subspace
$V$ of dimension $h$.
We consider the short exact sequence 
\[\cO_X \otimes V^\vee \into \widetilde E \onto E \]
as in Lemma \ref{lem:2ses}.
By the Lemma, $\widetilde E$ is $\sigma_+$-stable. Clearly $\left(\widetilde E\right)_{\mathrm{red}}
= E$ and $H^0(\widetilde E) = V^\vee$. Therefore, $\widetilde E\in M^h_{k, \kred}$, proving the
non-emptiness as claimed.

Our construction depended on a choice of a point in the Grassmannian bundle over $M^0_{\kred, \kred}$ whose
fiber over $E$ is given by the set of $h$-dimensional (or, equivalently, $D$-dimensional) subspaces in $\Ext^1(E, \cO_X) \cong \C^{-\chi}$.
On the other hand, applying Lemma \ref{lem:2ses}, part \ref{enum:sesW}, we see that any object
$F \in M^h_{k, \kred}$ fits into a short exact sequence
\[ \cO_X \otimes H^0(F) \into F \onto F_{\mathrm{red}} \]
with $F_{\mathrm{red}} \in M^0_{\kred, \kred}$, and thus this Grassmannian bundle describes the
entire stratum $M^h_{k, \kred}$.  Its dimension is
\begin{align*}
\dim M^0_{\kred} + \dim Gr(k-k_{\mathrm{red}},-\chi)
& =  w^2_0 +2 -2\kred\chi+(k-k_{\mathrm{red}})(-\chi-k+k_{\mathrm{red}}) \\
& =  w^2_0 +2 -2k\chi+(k-k_{\mathrm{red}})(\chi-k+k_{\mathrm{red}}) \\
& =  w^2_0 +2 -2k\chi-(h-\chi)h  \\
& =d(k,h). 
\end{align*}

For $k-k_{\mathrm{red}} \ge -\chi$, we proceed by induction on $k$. 
Let $h$ be the integer as in the statement of the Proposition.   Let $h'$ be the maximum integer, such that 
\[\Delta_{KLM}\left(\frac{h'}{-\chi}\right)\leq \frac{k-h-k_{\mathrm{red}}}{-\chi}  .\]
By the functional equation of $\Delta_{KLM}$, $h'\geq h+\chi\geq 0$.  By induction on $k$, there is an object
$E$ in $M^{h'}_{k-h,k_{\mathrm{red}}}$. By Lemma \ref{lem:extdim-new} the dimension of
$\Ext^1(E, \cO_X)$ is $h' - \chi \ge h$, and thus we can choose a subspace $V$ of dimension $h$. 
We again consider the extension 
\[\cO_X \otimes V^\vee \into \widetilde E \onto E; \]
applying Lemma \ref{lem:2ses} as before shows that $\tilde E$ is $\sigma_+$-stable with
$\left(\widetilde E\right)_\mathrm{red} \cong \ere$. In particular, $\tilde E\in M^{h''}_{k,k_{\mathrm{red}}}$ for
some $h'' \ge h$. 
Now since $h$ is the maximum number satisfying (\ref{thk}), Proposition \ref{mhtkn} says that
$M^{h''}_{k, \kred}$ is empty for $h'' > h$. Hence $h'' = h$, and so $M^{h}_{k,k_{\mathrm{red}}}$ is non-empty as
claimed.

When equality holds, it remains to show the statement about the dimension. In that case, we have $h'
= h + \chi$ in each induction step, and in particular the remainder $D$ remains preserved at each step. We have a bijective morphism 
\begin{align*}
\mhkkred & \rightarrow M^{h+\chi}_{k-h,k_{\mathrm{red}}} \\ 
E & \mapsto E' = E/\left(\cO_X\otimes \Hom(\cO,E)\right);
\end{align*}
indeed, its inverse is given by associating to $E$ the extension $\widetilde E$ given by 
\[
\cO_X \otimes \Hom(E,\cO_X[1])^\vee \into \tilde E\onto E.
\]
It follows by induction that $M^h_{k, \kred}$ is a Grassmannian-bundle over $M^0_{\kred}$ of
dimension
\begin{align*}
d(k-h, h + \chi) & = 
w_0^2 - 2(k-h)\chi + 2 + (h + \chi)\chi - (h+\chi)^2  \\
& = 
w_0^2 - 2k\chi + 2 + h \chi - h^2 = d(k,h).
\end{align*}
\end{proof}

\begin{Lem}
Let $h, k, k_{\mathrm{red}}\in\Z_{\geq 0}$. Whenever  $M^h_{k,k_{\mathrm{red}}}$ is non-empty, it satisfies
\[ \dim M^h_{k,k_{\mathrm{red}}} \le d(k,h). \]
Moreover, if equality holds, then  $h, k, k_{\mathrm{red}}$ satisfy equality in equation \eqref{thk}. 
\label{dimmhkt2}
\end{Lem}
\begin{proof}
We prove the statement by induction on $k$, the case $k=0$ being obvious. Similarly, for $h =0$, we have $M^0_{k,k}=M^0_k$, which is an open subset of $M_{\sigma_+}(w_k)$ by Proposition \ref{prop:nonemptyMHkkred}.

We thus assume $k > 0$ and $h > 0$. Given $E \in M^h_{k,k_{\mathrm{red}}}$, again consider the quotient $E'$ (in $\cP_0$) defined by
\[ \cO_X \otimes H^0(E) \into E \onto E', \]
which is $\sigma_+$-stable with $h^0(E') \ge h + \chi$ by Lemma \ref{lem:2ses}. We consider the map 
\[ \pi \colon M^h_{k, \kred} \to \bigcup_{l \ge \max(0, h+\chi)} M^l_{k-h, \kred}, \quad E \mapsto E', \]
where the right-hand-side is considered as a locally closed subset of $M_{\sigma_+}(w_{k-h})$.

It is sufficient to prove the dimension bound for each of the pre-images of the finitely many strata on the right-hand-side.
 By the induction on $k$,
the dimension of $M^l_{k-h,k_{\mathrm{red}}}$ does not exceed $d(k-h,l)$. 
For each $E'\in M^l_{k-h,k_{\mathrm{red}}}$, we have $\dim \Ext^1(E', \cO_X) = l - \chi$ by Lemma
\ref{lem:extdim-new}. Therefore, $\pi^{-1}(E')$ is a locally closed subset of the Grassmannian of
$h$-dimensional subspaces in $\C^l$ (defined by the associated extension having $h$-dimensional space of global sections); in particular its dimension is at most $h (l - \chi - h)$.

Thus, if the pre-image of $M^l_{k-h, \kred}$ is non-empty, its dimension is bound by
\begin{align}
\dim \pi^{-1} \left(M^l_{k-h, \kred}\right) \leq
& \dim M^l_{k-h,k_{\mathrm{red}}}+h(l-\chi-h) \nonumber \\
\leq & d(k-h, l) + h(l-\chi-h) \nonumber \\
= & d(k,h)+2h\chi-l(l-\chi)+h(h-\chi)+h(l-\chi-h) \nonumber \\
= & d(k,h) -l (l-\chi-h) \nonumber \\
\leq & d(k,h). \label{eq:ineqchain}
\end{align}
The last inequality achieves equality only if $l=0\geq h+\chi$ or $l=h+\chi> 0$.

In the former case, $M^0_{k-h,k_{\mathrm{red}}}$ is non-empty only if $k_{\mathrm{red}}=k-h$. Since $h+\chi\leq 0$, we have
$\Delta_{KLM}\left(\frac h{-\chi}\right) = \frac h{-\chi} = \frac{k - \kred}{-\chi}$,
and thus equality in \eqref{thk}.

In the
second case $l = h+\chi$, by the induction on $k$, the second inequality above can only be an equality if $h+\chi$, $k-h$, $k_{\mathrm{red}}$ satisfy equality in \eqref{thk}. By the functional equation of $\Delta_{KLM}$, we conclude
\[ \Delta_{KLM}\left(\frac h{-\chi}\right) 
= \Delta_{KLM}\left(\frac {h + \chi}{-\chi}\right) + \frac h{-\chi}
= \frac{k - h - \kred}{-\chi} + \frac h{-\chi} = \frac{k-\kred}{-\chi}
\]
as claimed.
\end{proof}
\def\mghk{M^{\geq h}_k}

\begin{Thm} \label{thm:maintechnical}
The moduli space $\mhk$ is nonempty if and only if 
\[\frac{k}{-\chi}\geq \Delta_{KLM}\left(\frac{h}{-\chi}\right).\]
Whenever non-empty, $\mhk$ is irreducible with  expected dimension 
$d(k,h)$.
\end{Thm}
\begin{proof}
Observe either by induction, or by the explicit formula \eqref{eq:Deltaexplicit}, that
$\chi \cdot \Delta\left( \frac h{-\chi}\right)$ is an integer. Therefore, given $h$ and $k$
satisfying the assumptions of the Theorem, there is a unique integer $k_0$ with $0 \le k_0 \le
k$ such that equality holds in equation \ref{thk}. By Proposition \ref{prop:nonemptyMHkkred}, for
this choice of $\kred$, the stratum $M^h_{k, \kred}$ is non-empty, irreducible and of expected
dimension $d(k, h)$. 

On the other hand, for all other choices of $\kred$, Lemma \ref{dimmhkt2} shows that the stratum
$\mhkkred$ has dimension strictly smaller than $d(k, h)$. By Lemma \ref{lem:dkh}, it is thus
contained in the closure of a stratum of bigger dimensions.

Conversely, if 
$ \frac{k}{-\chi} < \Delta_{KLM}\left(\frac{h}{-\chi}\right)$, then Proposition \ref{mhtkn} shows
that $M^h_k$ is empty.
\end{proof}

Note that all constructions in this sections have been invariant under the action of $X$ by translation. We note the following consequence of this fact when combined with the proof of Theorem \ref{thm:maintechnical}:
\begin{Prop} \label{prop:grassmannbundle}
Assume that $M^h_k$ is non-empty, and let $k_0$ with $0 \le k_0 \le k$ be such that
\[ \frac{k-k_0}{-\chi} = \Delta_{KLM}\left(\frac h{-\chi}\right). \]
Let $0 \le D < -\chi$ be the remainder of division of $h$ by $-\chi$. Then $M^h_k$ has a dense open subset $M^h_{k, k_0}$, invariant under the action of $X$ by translation, that is isomorphic to a bundle over $M^0_{k_0}$ of Grassmannians of type $\mathrm{Gr}(D, -\chi)$. Its fiber over
a point $F\in M^0_{k_0}$ parameterises iterated extensions of $F$ by $\cO_X$.
\end{Prop}

\section{Conclusion}
\label{sect:deduce}

We will now prove our main result, which we first recall:
\begin{repThm}{main1}
Assume $\chi \neq 0$.
The Brill-Noether locus $V^r_d(\abs{H})$ is non-empty if and only if 
\begin{equation} \label{mainineq}
\rho + g - 2 \ge D\abs{\chi}-D^2, 
\end{equation} 
where $D$ denotes the remainder of division of $r+1$ by $\abs{\chi}$.
Moreover, when it is non-empty, it is generically smooth and of expected dimension $\rho + g - 2$.

When the above inequality is strict, then $\VrdH$ is irreducible. Otherwise, it is a disjoint union of Grassmannians.
\end{repThm}

\begin{proof}
Using the derived dual on $X$, or, equivalently, Serre duality on the individual curves, we first
reduce to the case $\chi < 0$.

Thus we are in the situation of the previous section. 
By Lemma \ref{prop:gwall}, we have $M_H(0,H,\chi)=M_{\sigma_+}^s(0,H,\chi) = M_{\sigma_+}(w_R)$.
(Note that by Definition \ref{def:wk}, the value $R$ is non-negative by $\chi<0$.)

Hence $V^r_d(\abs{H})$ is the intersection of the Brill-Noether variety $M^{r+1}_R$ studied in the
previous section with the set of sheaves supported on a curve in the linear system $\abs{H}$ (rather
than a translate of such a curve). 
If we consider the action of $X$ on $M_H(0, H, \chi) = M_{\sigma_+}(w_R)$ by translation, then
 each orbit contains finitely many sheaves
supported on a curve in the linear system $\abs{H}$.
Moreover, the action leaves the Brill-Noether variety $M^{r+1}_{R}$ invariant. 

It follows that $\lrd$ is non-empty if and only if $M^{r+1}_R$ is non-empty, with
\begin{align*}
\dim \lrd &  = \dim M^{r+1}_R - 2 = d(R,r+1)-2 \\ 
&  = (0,H,\chi)^2+2+(r+1)\chi-(r+1)^2-2 \\
& = 2g - 2 - (r+1)(r+1-\chi) = \rho +g-2.
\end{align*}
Let $r+1=s(-\chi)+D$, where $s\in \Z_{\geq 0}$, $0\leq D<\chi$. By Theorem \ref{thm:maintechnical}, $M_{R}^{r+1}$ is non-empty if and only if 
\begin{align*}
\frac{R}{-\chi} & \geq \left(\frac{r+1}{-\chi}-\frac{1}{2}\floor*{\frac{r+1}{-\chi}}\right)\left(\floor*{\frac{r+1}{-\chi}}+1\right)\\
\Longleftrightarrow R  & \geq  \frac{1}{2}(r+1+D)(s+1).
\end{align*}
The bound on the right is an integer: indeed, when $s$ is even, $r+1$ and $D$ have the same parity.
Thus, we may replace $R$ by $\frac{H^2}{-2\chi}$, omitting the round-down (see Definition \ref{def:wk}).
Therefore, $M_{R}^{r+1}$ and $\lrd$ are non-empty if and only if 
\begin{eqnarray*}
2g-2-(r+1+D)(r+1-D-\chi)& \geq & 0 \\ 
\Longleftrightarrow
2g - 2 - (r+1)(r+1 - \chi) + D^2 + D \chi & \ge & 0 \\
\Longleftrightarrow 
\rho + g-2 & \ge & D(-\chi) - D^2.
\end{eqnarray*}
Now assume that the inequality \eqref{mainineq} is strict; we need to show that $\VrdH$ is irreducible. Note that we can define $\VrdH$ equivalently as the preimage
of $0 \in \hatX$ under the map
\begin{equation} \label{eq:defD}
D \colon M^{r+1}_R \to \hatX, \quad E \mapsto \det(E) \otimes \cO(-H). \end{equation}
Now let $k_0$ be as in Proposition \ref{prop:grassmannbundle}, and consider the open subset $M^{r+1}_{R, k_0} \subset M^{r+1}_R$. Since this inclusion is translation-invariant, it suffices to show that the  restriction factors as
\[ M^{r+1}_{R, k_0} \to M^0_{k_0} \xrightarrow{\underline{\det} \times \underline{\widehat{\det}}} \hatX \times X \to \hatX; \]
here the first map is the projection to the base of the Grassmannian bundle given in the Proposition \ref{prop:grassmannbundle},  the second map is the restriction of the Albanese map appearing in Proposition \ref{prop:albanese}, and the third one is, up to shift, given by projection to the first factor.

We claim that under our assumptions, each of these maps has connected fibers. Indeed, this is obvious for the first and the third map. For the second map, note that by Proposition \ref{prop:grassmannbundle}, the dimension of $M_{k_0}$ must be equal to 
\[ \dim M_{k_0} = \dim M_R^{r+1} - \dim \mathrm{Gr}(D, -\chi) 
= \rho + g - D(-\chi - D). \]
Thus, if the inequality \eqref{mainineq} is strict, then $M_{k_0}$ has dimension at least 4, and we can apply 
Proposition \ref{prop:albanese} to deduce connectedness of fibers for $\underline{\det} \times \underline{\widehat{\det}}$.

The composition is thus a map between smooth irreducible varieties with connected fibers; since the generic fiber is smooth and connected, it is irreducible. Since translation acts transitively on the set of fibers, it follows that $M^{r+1}_{R, k_0}$ is smooth and irreducible.

Otherwise, if equality holds in \eqref{mainineq}, then $M_{k_0}$ is two-dimensional, i.e.~$w_{k_0}^2 = 0$. Hence we have 
$-k_0 = \rk(w_{k_0}) = \frac{H^2}{2\chi} = \frac{g-1}{\chi}$. By Proposition
\ref{prop:albanese}, the moduli space $M_{k_0}$ contains $k_0^2$ many elements with determinant
$\cO(H)$ (rather than one of its translates). The variety $\VrdH$ consists of one Grassmannian for each of these objects. 
\end{proof}

\begin{Rem}
Since $0 \le D < -\chi$, the right-hand-side of \eqref{mainineq} is always positive, as we should expect.
On the other hand, from $r+1 \ge D$ one can verify that if $\rho$ violates \eqref{mainineq}, then in fact
$\rho < 0$.
\end{Rem}

\begin{Rem}
This inequality gives exactly the same bound as that in Theorem A.1 in \cite{KLM}:
\begin{align}\rho+r(r+2)\geq -\floor*{\frac{r}{-\chi}}\chi\left(r+1+\frac{1}{2}\chi\left(\floor*{\frac{r}{-\chi}}+1\right)\right).
\label{klmbd}
\end{align}
\label{rem:klmbd}
\end{Rem}
When $\floor*{\frac{r}{-\chi}}=\floor*{\frac{r+1}{-\chi}}=s$, the inequality (\ref{klmbd})  is equivalent to 
\begin{align*}
\rho+r(r+2) & \geq -s\chi(r+1-\frac{1}{2}(-\chi)(s+1)) \\ 
\Longleftrightarrow
2\rho+2r(r+2)& \geq 2(r+1-D)(r+1- \frac{1}{2}(r+1-D)+\frac{1}{2}\chi)\\ 
\Longleftrightarrow
\rho +g +(r+1)(\chi+r+1)-2& \geq (r+1-D)(r+1+D+\chi) \\ 
\Longleftrightarrow
\rho+g-2& \geq D(-\chi)-D^2. \\ 
\end{align*}

In the remaining case $r+1=s(-\chi)$, the inequality (\ref{klmbd})  is equivalent to 
\begin{align*}
\rho+r(r+2) & \geq (r+1+\chi)(r+1+\frac{1}{2}\chi s) \\ 
\Longleftrightarrow
2\rho+2r(r+2)& \geq (r+1+\chi)(r+1)\\ 
\Longleftrightarrow
\rho +g -2& \geq 0.  
\end{align*} \hfill $\square$

\begin{proof}[Proof of Theorem \ref{mainthm2}]
When $k<-R$, we have $\vv^2<0$. The inequality fails since $\vv^2-(r+1)(r+1-\chi)<0\leq  D(-\chi)-D^2$. The moduli space $M^{r+1}_H(\vv)$ is empty due to Proposition \ref{prop:mv0square0}.

When $k\geq -R$, by Lemma \ref{prop:gwall}, we have $M_H(k,H,\chi)=M_{\sigma_+}^s(k,H,\chi) = M_{\sigma_+}(w_{k+R})$. It follows that $M^{r+1}_H(\vv)=M^{r+1}_{k+R}$.  We apply Theorem \ref{thm:maintechnical} for the space $M_{k+R}^{r+1}$. This space is non-empty if and only if 
\begin{align*}
\frac{k+R}{-\chi} & \geq \left(\frac{r+1}{-\chi}-\frac{1}{2}\floor*{\frac{r+1}{-\chi}}\right)\left(\floor*{\frac{r+1}{-\chi}}+1\right)\\
\Longleftrightarrow R  & \geq  \frac{1}{2}(r+1+D)\left(\frac{r+1-D}{-\chi}+1\right)-k.
\end{align*}
Note that the bound on the right is an integer, we may substitute this bound to the constraint on $R$. The space  $M_{k+R}^{r+1}$ is non-empty if and only if 
\begin{align*}
2g-2-(r+1+D)\left(\frac{r+1-D}{-\chi}+1\right)(-\chi)+2k\chi & \geq 0 \\ 
\Longleftrightarrow
\vv^2+(r+1)\chi-(r+1)^2& \geq D(-\chi)-D^2.
\end{align*}
By Theorem \ref{thm:maintechnical}, when $M^{r+1}_{R+k}$ is non-empty, it is irreducible of the expected dimension:
\[d(k+R,r+1)=\vv^2+2-(r+1)(r+1-\chi).\]
\end{proof}

\section{Generalisation}
\label{sect:moregeneral}

Let $(X, H)$ be a polarized abelian surface satisfying Assumption \star.
In this section, we explain how to adapt all our arguments from abelian Picard rank one to $X$.

 Let $\Lambda_H \cong \Z^3$ denote the image of the map
\[ v_H \colon K(X) \to \R^3, \quad E \mapsto \left(\ch_0(E), H.\ch_1(E), \ch_2(E) \right).\]
We will only consider stability conditions for which the central charge factors via $v_H$, and denote the space of such stability conditions by $\Stab_H(X)$. The pair $\sigma_{\alpha, \beta} := \left(\Coh^\beta X, Z_{\alpha, \beta}\right)$ defines a stability condition on $D^b(X)$  and there is a continuous map from $\R_{>0} \times \R \to \Stab(X)$. The only difference here is that $\Stab_H(X)$ is not the whole space but a slice of the whole space. The slope function $\nu_{\alpha, \beta}$ is defined in the same way. 

The Mukai pairing equips $\Lambda_H\otimes \R$ with the quadratic form $Q$, and identifies the upper half plan $\R_{>0} \times \R$
with the projectivization of the negative cone of $Q$, via taking the kernel of $\Zab$ in $\Lambda_H\otimes \R$. All the propositions in Section \ref{sect:background} hold for the higher Picard rank case.

In the proof of Proposition \ref{prop:gwall}, we use the fact that `$\Im Z_{\alpha, 0}(E)$ is of the form $\Z_{\ge 0} \cdot H^2$ for all $E \in \Coh^\beta X$'. This also holds by Assumption (*). For  Lemma \ref{Owall}, we shall modify the statement  to `$v_H(E_i)$ is either $v_H(w_{k'})$ or $(r,0,0)$', and prove it in the same way as the Picard number one case. Similarly, for Corollary \ref{cor:ston}, we change the condition to `for any $\sigma_0$-stable object $O$ such that $v_H(O)=(r,0,0)$'. Note that this modification does not affect the argument for Lemma \ref{lem:ext}, which is the only place that uses Corollary \ref{cor:ston}. All the other statements do not rely on the Picard rank.

\bibliography{all}                      
\bibliographystyle{halpha}     

\end{document}